\documentclass{amsart}

\usepackage{amsmath}
\usepackage{amsthm}
\usepackage{amsfonts}
\usepackage{amssymb}
\usepackage{bbm}
\usepackage{tikz}
\usepackage{pgf}
\usepackage{graphicx}
\usepackage{fancyhdr}
\usepackage{blkarray}
\usepackage{multirow}
\usepackage{xypic}
\usepackage[margin = 1in]{geometry} 

\DeclareMathOperator{\Forall}{\;\forall}
\DeclareMathOperator{\st}{\,|\,}
\DeclareMathOperator{\im}{\text{im}}

\DeclareMathOperator{\C}{\mathbb{C}}
\DeclareMathOperator{\R}{\mathbb{R}}
\DeclareMathOperator{\Z}{\mathbb{Z}}
\DeclareMathOperator{\Q}{\mathbb{Q}}
\DeclareMathOperator{\N}{\mathbb{N}}

\DeclareMathOperator{\Prin}{Prin}
\DeclareMathOperator{\Cart}{Cart}
\DeclareMathOperator{\Pic}{Pic}
\DeclareMathOperator{\Div}{Div}
\DeclareMathOperator{\link}{link}
\DeclareMathOperator{\Star}{star}
\DeclareMathOperator{\diag}{diag}
\DeclareMathOperator{\Diag}{Diag}
\DeclareMathOperator{\QCart}{\mathbb{Q}Cart}

\newtheorem{theorem}{Theorem}
\numberwithin{theorem}{section}

\newtheorem{proposition}[theorem]{Proposition}

\newtheorem{conjecture}[theorem]{Conjecture}

\theoremstyle{definition}
\newtheorem{definition}[theorem]{Definition}
\newtheorem{example}[theorem]{Example}
\newtheorem{remark}[theorem]{Remark}

\title{A Chip-Firing Game on the Product of Two Graphs and the Tropical Picard Group}
\author{Alexander Lazar}

\begin{document}

\maketitle

\section{Introduction}
The chip-firing game is a well-studied subject in combinatorics with deep connections to other areas of mathematics, including algebraic geometry. In \cite{Graph_Riemann_Roch}, Baker and Norine demonstrated an analogy between the chip-firing game and linear equivalence of divisors on a Riemann surface. In the same paper, they proved a graph-theoretic analogue of the Riemann-Roch theorem from algebraic geometry. In fact, beyond these analogies, there is a fundamental connection between graph theory and algebraic geometry.

For example, let $X_t$ be a family of smooth curves over $\C$, parametrized by $t$, that becomes a singular curve when $t=0$ (this is an instance of a ``degeneration'' of curves). We can associate a graph to $X_0$, and divisors on $X_0$ descend to chip-firing states on that graph, with chip-firing moves on the graph corresponding to linear equivalence of divisors on $X_0$. Thus, one can address questions in algebraic geometry by studying the chip-firing game.

Furthermore, one can generalize this sort of reasoning to higher-dimensional objects. In his preprint \cite{Cartwright_Paper}, Cartwright introduced the notion of a weak tropical complex in order to generalize the concepts of divisors and the Picard group on graphs from \cite{Graph_Riemann_Roch}. A weak tropical complex $\Gamma$ is a $\Delta$-complex equipped with algebraic data that allows $\Gamma$ to be viewed as the dual complex of a particular kind of degeneration over a discrete valuation ring. 

Within the context of weak tropical complexes, the analogue of the chip-firing game is the theory of divisors. A divisor on a weak tropical complex is a formal linear combination of codimension-$1$ polyhedral subsets (which we can think of as a higher-dimensional ``chip configuration''), and two divisors are linearly equivalent if they differ by a ``tropical principal divisor'' (which we can think of as encoding a chip-firing move). One important invariant in this theory is the tropical Picard group, which consists of a certain set of tropical divisors up to linear equivalence.

Every finite graph has a unique tropical complex structure. In this tropical complex structure, divisors correspond to states in a certain form of the chip-firing game on that graph. If $G$ and $H$ are graphs, and $\Gamma$ is a triangulation of their product obtained by adding in a diagonal of each resulting square, then $\Gamma$ has a weak tropical complex structure that is compatible with the tropical complex structures on $G$ and $H$. Motivated by analogous results in algebraic geometry (e.g., \cite[Theorem 1.7]{Ischebeck}), Cartwright conjectured that the Picard groups of $\Gamma$, $G$, and $H$ were closely related.

\begin{theorem}[Main Theorem]\label{MainTheorem}Let $\Pic(\Gamma)$ be the tropical Picard group of $\Gamma$, and $\Pic(G)$ and $\Pic(H)$ be the tropical Picard groups of $G$ and $H$. Then, there is a map $\gamma: \Pic(G) \times \Pic(H) \to \Pic(\Gamma)$ that is always injective and is surjective if at least one of $G$ or $H$ is a tree.\end{theorem}

As we shall see, the proof of this theorem is independent of the choices made in constructing $\Gamma$, although the cokernel of the map $\gamma$ may vary as the triangulation changes. 

In this paper, we prove a seemingly weaker form of the conjecture where we restrict our attention to ``ridge divisors'' --- divisors that are formal linear combinations of ridges of $\Gamma$. Due to computations in sheaf cohomology (see \cite[Section 3]{Cartwright_Surfaces}), the ridge divisor form of the conjecture implies the more general form. Throughout the rest of the paper we will only consider ridge divisors, so we will omit the word ``ridge'' and simply use the term ``divisor''.

The structure of this paper is as follows: in Section 2, we provide preliminary definitions and notation; in Section 3 we prove some technical results about divisors on the product of graphs; and in Section 4 we prove the main theorem of the paper. Finally, we make the following conjecture:

\begin{conjecture}\label{MyConjecture}$\Pic(\Gamma) \cong \Pic(G) \times \Pic(H) \times \Z^{g(G)g(H)}$. \end{conjecture}

\textbf{Acknowledgements:} This work was completed as part of the requirements for earning a Master's degree from the University of Kansas. The conjecture resolved in this paper was presented at the American Institute of Mathematics workshop ``Generalizations of chip-firing and the critical group'' (July 8th-12th, 2013): http://aimath.org/pastworkshops/chipfiring.html.

I would like to thank my M.A. thesis committee: Marge Bayer, Yunfeng Jiang, and Jeremy Martin. I would especially like to thank my M.A. advisor, Jeremy Martin, for his invaluable help in writing the thesis that eventually became this paper. I would also like to thank James McKeown and Alexander Schaefer for their helpful comments.

\section{Preliminaries}

If $\Gamma$ is a pure $n$-dimensional $\Delta$-complex (in the sense of \cite{Hatcher_topology}), we write $\Gamma_k$ for the $k$-skeleton of $\Gamma$. Faces of $\Gamma$ of dimension $n$ are called \textbf{facets}, and faces of dimension $n-1$ are called \textbf{ridges}. We write $V(\Gamma)$ and $E(\Gamma)$ for the sets of vertices and edges of $\Gamma$, respectively.

The following definition is due to Cartwright \cite[Definition 2.1]{Cartwright_Paper}.

\begin{definition}\label{tropcompdef} An $n$-dimensional \textbf{weak tropical complex} is a pure connected $n$-dimensional $\Delta$-complex $\Gamma$, along with a function $\alpha: \Gamma_{n-1} \times V(\Gamma) \to \Z$ such that for every ridge $r$, $$\sum_{v\in r}\alpha(r,v) = \text{deg}(r),$$ where $\text{deg}(r)$ is the number of $n$-faces of $\Gamma$ containing $r$, and where $\alpha(r,v) = 0$ if $v\notin r$.

\end{definition}

\subsection{Divisors on Tropical Complexes}\label{DivisorSubsection}\leavevmode

For the rest of this section, let $(\Gamma,\alpha)$ be a weak tropical complex. A \textbf{divisor} on $\Gamma$ is a formal $\Z$-linear combination of the ridges of $\Gamma$. If $C$ is a divisor on $\Gamma$ and $r$ is a ridge, we write $C(r)$ for the coefficient of $r$ in $C$.

A \textbf{piecewise linear (PL)} function on $\Gamma$ is a continuous piecewise linear function $\phi$ that restricts to a linear function with integer slope on each simplex in $\Gamma$. We can associate to each PL function $\phi$ a divisor $\Div(\phi)$ as follows: 

\begin{equation}\label{divisorequation}\Div(\phi) = \sum_{\text{ridges } r}\left(\sum_{\text{facets } f \supseteq r}\phi(f\setminus r) - \sum_{v \in r}\alpha(r,v)\phi(v)\right)[r].\end{equation}

\begin{definition}\label{tropprindef}A \textbf{tropical principal divisor} is a divisor that can be written as $\Div(\phi)$ for some PL function $\phi$. We define $\Prin(\Gamma)$ to be the group of principal divisors on $\Gamma$.\end{definition}

For any vertex $v \in \Gamma$, we define $\phi_v$ to be the unique PL function that is $1$ on $v$ and $0$ on all other vertices of $\Gamma$. We note that $\Prin(\Gamma)$ is generated by $\{\Div(\phi_v) \st v \in V(\Gamma)\}$, since our PL functions are uniquely specified by their values on $V(\Gamma)$.

We are interested in divisors that are ``locally principal'', which we make precise in the following sense:

\begin{definition}\label{tropcartdef}A \textbf{tropical Cartier divisor} is a formal $\Z$-linear combination $D$ of ridges of $\Gamma$ such that for every $v \in V(\Gamma)$, there exists a PL function $\phi$ such that $D$ and $\Div(\phi)$ agree on all ridges containing $v$. We let $\Cart(\Gamma)$ be the group of Cartier divisors on $\Gamma$.\end{definition}

\begin{definition}\label{tropqcartdef}A \textbf{tropical $\Q$-Cartier divisor} is a divisor $D$ such that $mD$ is Cartier for some $m \in \Z\setminus \{0\}$. We let $\QCart(\Gamma)$ be the group of $\Q$-Cartier divisors on $\Gamma$.\end{definition}

Note that $\Prin(\Gamma) \subseteq \Cart(\Gamma) \subseteq \QCart(\Gamma)$.

\begin{definition}\label{troppicdef}The \textbf{tropical Picard group} $\Pic(\Gamma)$ of $\Gamma$ is the quotient $\Cart(\Gamma)/\Prin(\Gamma)$.\end{definition}
If two divisors differ by a principal divisor, they are said to be \textbf{linearly equivalent}. Thus, $\Pic(\Gamma)$ is the group of linear equivalence classes of Cartier divisors.

Since we are working in dimensions $1$ and $2$, we make the following definition (see \cite[p. 7]{Cartwright_Surfaces}):

\begin{definition} The \textbf{tropical divisor class group} $\text{Cl}(\Gamma)$ of $\Gamma$ is the quotient $\QCart(\Gamma)/\Prin(\Gamma)$.\end{definition} 

\textbf{Note.} As alluded to in the introduction, the definitions in this section (including that of a piecewise linear function on a weak tropical complex) are more restrictive than those in \cite{Cartwright_Paper}. In \cite{Cartwright_Paper}, the divisors above are called \textbf{ridge divisors}, and the group of Cartier ridge divisors modulo principal ridge divisors is denoted $\Pic_{\text{ridge}}(\Gamma)$. Since we are only dealing with ridge divisors in this paper, we omit the word ``ridge''.

\begin{example}
Consider the following two-dimensional complex $\Gamma$, with vertex set $\{v_1,v_2,v_3,v_4\}$ and edge set $\{e_1,e_2,e_3,e_4,e_5\}$.
\begin{center}\includegraphics{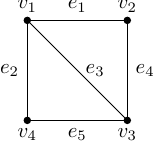}\end{center}

We put a weak tropical complex structure $\alpha$ on $\Gamma$, which we write in the form of a matrix $A$ whose $(i,j)$th entry is $\alpha(e_i,v_j)$:
$$A = \left[\begin{array}{cccc}0 & 1 & 0 & 0\\
0 & 0 & 0 & 1\\
1 & 0 & 1 & 0\\
0 & 1 & 0 & 0\\
0 & 0 & 0 & 1
\end{array}\right].$$

We recall that any PL function $\phi$ is uniquely specified by its values on the vertices of $\Gamma$. Thus, Equation (\ref{divisorequation}) tells us that we can write the homomorphism $\Div: \Z^{V(\Gamma)} \to \Z^{E(\Gamma)}$ as the following matrix $D$:

$$D = \left[\begin{array}{cccc} 0 & -1 & 1 & 0\\
0 & 0 & 1 & -1\\
-1 & 1 & -1 & 1\\
1 & -1 & 0 & 0\\
1 & 0 & 0 & -1\end{array}\right].$$

The group $\Prin(\Gamma)$ consists of the column-span of the matrix $D$. The divisor $C = [e_1] + 2[e_2] + 3[e_4] + [e_5]$ is Cartier but non-principal, since it is not in the column span of $D$. Finally, Theorem \ref{AltCart} implies that in this case all divisors are $\Q$-Cartier.

\end{example}

\subsection{Divisors on Graphs}\leavevmode

Of particular interest is the case of one-dimensional (weak) tropical complexes, i.e., graphs. 

Let $G$ be a loopless connected graph, possibly with multiple edges between a given pair of vertices, with vertex set $V(G)$ and edge set $E(G)$. Restricting Definition \ref{tropcompdef} to dimension $1$, a $1$-dimensional weak tropical complex structure on $G$ is a function $\alpha: V(G) \times V(G) \to \Z$ such that for all $v,w \in V(G)$,
$$\alpha(v,w) = \begin{cases}\deg(v), & v=w\\ 0, & v\neq w\end{cases}.$$

Thus, we see that $G$ admits exactly one tropical complex structure. Furthermore, if we take the PL function $\phi_v$ as in Section $\ref{DivisorSubsection}$, we see that
$$\Div(\phi_v) = \left(\sum_{w \in V(G)}\text{adj}(v,w)[w]\right) - \deg(v)[v],$$
where $\text{adj}(v,w)$ is the number of edges between $v$ and $w$. Since $\{\phi_v \st v \in V(G)\}$ spans the module of PL functions on $G$, we can express the map $\Div$ as a $|V(G)| \times |V(G)|$ matrix $L(G)$, with
$$L(G)_{v,w} = \begin{cases}-\deg(v), & v=w\\ \text{adj}(v,w), & v\neq w\end{cases}.$$ This matrix is exactly the \textbf{Laplacian matrix} of $G$.

A Cartier divisor on $G$ is a divisor $C$ such that, for all $v \in V(G)$, there is some principal divisor $D_v$ with $D_v(v) = C(v)$. This condition is always true, so every divisor on $G$ is Cartier.

Since $\Prin(G)$ is exactly the column-span of $L(G)$, we see that $\Pic(G) \cong \text{coker}(L(G))$. Therefore, by the Matrix-Tree Theorem, $\Pic(G) \cong \Z \oplus K(G)$, where $K(G)$ is a finite group called the \textbf{critical group} of $G$.

\begin{remark} Divisors on a graph are equivalent to positions in a certain formulation of the well-known chip-firing game on the graph $G$ \cite[Lemma 4.3]{Graph_Riemann_Roch}. In this case, principal divisors are precisely those positions that can be reached from the configuration where all vertices have no chips \cite[Lemma 4.3]{Graph_Riemann_Roch}.\end{remark}

\section{Balancing Conditions}

For the rest of this paper, we assume without loss of generality that all graphs are connected. Given a pair of graphs $G$ and $H$, the product $G\times H$ is a \textbf{cubical complex} --- a cell complex whose $0$-cells come from pairs of vertices, whose $1$-cells come from vertex-edge pairs, and whose $2$-cells are squares arising from pairs of edges. For much of this paper, we will consider triangulations $\Gamma$ of $G\times H$, obtained subdividing each square into two triangles. We call the edges of the form (edge of $G$)$\times$(vertex of $H$) \textbf{horizontal edges}, edges of the form (vertex of $G$)$\times$(edge of $H$) \textbf{vertical edges}, and new edges added in this triangulation \textbf{diagonal edges}. We define $\Diag(\Gamma)$ to be the set of diagonal edges of $\Gamma$. If $\sigma$ is a square or a triangle, we define $\diag(\sigma)$ to be the unique diagonal edge contained in $\sigma$.

\begin{center}
\includegraphics{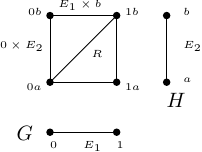}
\end{center}

In the preceding figure, $E_1 \times b$ is a horizontal edge, $0 \times E_2$ is a vertical edge, and $R$ is a diagonal edge.

We define a natural weak tropical complex structure on a triangulation $\Gamma$ of $G\times H$ as follows (this construction is due to Cartwright \cite{Cartwright_Email}; see also \cite[Example 6.2]{Cartwright_Paper_Old}). Let $v \in V(\Gamma)$ and $e \in E(\Gamma)$, and write $F(\Gamma)$ for the set of $2$-dimensional faces of $\Gamma$. Then, $$\alpha(e,v) =\begin{cases}1, & v\in e, \text{ and } e \in \Diag(\Gamma)\\ |\{\sigma \in F(\Gamma) : \; e\in \sigma, \; v\notin \diag(\sigma)\}|, & v \in e \text{ and } e \notin \Diag(\Gamma)\\ 0, & v\notin e\end{cases}.$$ 

The main result of this section is the following criterion for a divisor on $\Gamma$ to be $\Q$-Cartier. For any graph $G$ and any vertex $a$ of $G$, define $N_G(a)$ to be the set of neighbors of $a$. Note that this is a special form of the balancing condition mentioned in \cite[Section 5]{Cartwright_Paper}, but we prove it here for completeness.

\begin{theorem}[``Balancing Conditions'']\label{AltCart} Let $G$ and $H$ be graphs, and $\Gamma$ a triangulation of $G\times H$. A formal sum of ridges $D$ of $\Gamma$ is $\Q$-Cartier if and only if it satisfies the following conditions for all $v = (a,b) \in V(G\times H)$. First, for all $x \in N_G(a)$,

\begin{subequations}
\begin{equation}\label{balancing}\Xi^D_{ab}(x):= \sum_{\substack{c \in V(H)\\(xc,ab) \in E(\Gamma)}}\!\!\!\!\!\!\!\! D(xc,ab)\end{equation} 
\text{ is independent of the choice of } $x$.

\text{Second, for all }$y \in N_H(b)$,
\begin{equation}\Upsilon^D_{ab}(y):=\sum_{\substack{c \in V(G)\\(cy,ab) \in E(\Gamma)}}\!\!\!\!\!\!\!\! D(cy,ab)\end{equation}
\text{is independent of the choice of }$y$.

\end{subequations}\end{theorem}

\begin{proof}We recall that a tropical Cartier divisor $D$ is precisely one that is locally a principal divisor at every vertex $v$ in $G\times H$. By ``locally'' we mean that there is some principal divisor that agrees with the restriction of $D$ on the \emph{graph} star $E_{\Gamma}(v)$ of $v$ --- the union of the collection of edges containing $v$.

So, suppose we have a Cartier divisor $D$ on $\Gamma$. Fix an ordering $v_1, \dots, v_n$ of $V(\Gamma)$ and an ordering $e_1, \dots, e_m$ of $E(\Gamma)$. The set of principal divisors on $\Gamma$ is precisely the column span of the $(m\times n)$ matrix $M$ whose entries $i,j$ are given by: $$M_{ij} = \begin{cases} \alpha(e_i,v_j), & v_j \in e_i\\ -1, & e_i \in \link_{\Gamma}(v_j)\\ 0, & \text{otherwise}\end{cases}.$$

The matrix $M$ is the negation of the matrix of the map $\Div$ in Equation \eqref{divisorequation}. The negation does not affect the column span, and is more convenient for the purposes of this proof. 

Let $M_{v_k}$ be the submatrix of $M$ consisting of the rows of $M$ labeled by edges in $E_{\Gamma}(v_k)$. A divisor is locally principal at $v_k$ if its restriction to $E_{\Gamma}(v_k)$ is in the column span of $M_{v_k}$.

\textbf{Part 1:} Show that $\Q$-Cartier divisors satisfy the conditions of Theorem \ref{AltCart}.

Fix a vertex $v = ab \in \Gamma$, with $a \in G$ and $b \in H$. It suffices show that every column of $M_v$ satisfies the conditions of Theorem \ref{AltCart}, since a divisor $D$ is balanced if and only if its integer multiples are. Let $w$ be a vertex in $\Gamma$. The column of $M_v$ labeled by $w$ can be viewed as a divisor $D$ on $\Gamma$.

\underline{Case 1: $w=v$.} Every entry of $D$ is of the form $\alpha(e,v)$, where $v\in e$. Every edge containing $ab$ is of the form $(ab,xc)$. If $b = c$, $(ab,xc)$ is a horizontal edge; if $x=a$, it is a vertical edge; and if $a\neq x$ and $c\neq b$, it is a diagonal edge.

 Let $x\in N_G(a)$. In this case, $D(ab,xc) = \alpha((ab,xc),ab)$, so in particular $\alpha((ab,xc),ab) = 1$ when $b\neq c$. We write \begin{align*}\sum_{\substack{c \in V(H):\\ (ab,xc) \in E(\Gamma)}}D(ab,xc) &= D(ab,xb) + \sum_{\substack{ c \in V(H)\setminus b:\\ (ab,xc) \in E(\Gamma)}}D(ab,xc)\\ &= \alpha((ab,xb),ab) + |\{(ab,xc) \in E(\Gamma) \st c\neq b\}|.\end{align*} 

Now, $\alpha((ab,xb),ab) = \#\{\sigma \in F(\Gamma) : \; (ab,xb) \in \sigma, ab \notin \diag(\sigma)\}$. We note that every triangle containing the edge $(ab,xb)$ is either of the form $\{ab,xb,xc\}$ or of the form $\{ab,xb,ac\}$, with $c \in N_H(b)$. Triangles of the former type must include the diagonal edge $(ab,xc)$, while triangles of the latter type contain the diagonal edge $(ac,xb)$. In other words, $\alpha((ab,xb),ab)$ counts the number of triangles containing $(ab,xb)$ of the latter type, while every triangle of the former type containing $(ab,xb)$ gives rise to a diagonal edge containing $ab$. Thus, $$\alpha((ab,xb),ab) = |N_H(b)| -  |\{(ab,xc) \in E(\Gamma) \st c\neq b\}|,$$ so we see that \begin{align*}\sum_{\substack{c \in V(H):\\ (ab,xc) \in E(\Gamma)}}D(ab,xc) &= |N_H(b)| -  |\{(ab,xc) \in E(\Gamma) \st c\neq b\}| + |\{(ab,xc) \in E(\Gamma) \st c\neq b\}|\\ &= |N_H(b)|.\end{align*} The choice of $x$ was arbitrary, so $\Xi^D_{ab}(x)$ is independent of $x$.

\underline{Case 2: $w=a'b'$.} In this case, the edge $(ab,a'b')$ is a diagonal edge. Thus, we can compute $D$ explicitly: $$D(ab,xy) = \begin{cases} -1, & \{ab, xy, a'b'\} \in \Gamma\\ 1, & xy=a'b'\\ 0, & \text{otherwise}\end{cases}.$$

Fix $x \in N_G(a)$. If $x = a'$, then $$\Xi_{ab}^D(a')=\sum_{\substack{c \in V(H):\\ (ab,a'c) \in E(\Gamma)}} D(ab,a'c) = D(ab,a'b) + D(ab,a'b') + \sum_{c \in N_H(b)\setminus \{b'\}}D(ab,xc).$$ 

Since $(ab,a'b')$ is a diagonal edge in $\Gamma$, the horizontal edge $(ab,a'b)$ must be in $\Gamma$ as well. This means that the triangle $\{ab,a'b,a'b'\}$ is in $\Gamma$. Thus, $ D(ab,a'b) = -1$. By definition, $D(ab,a'b') = 1$. Finally, we note that for $c \in N_H(b)\setminus\{b'\}$, $D(ab,a'c) = 0$, since if $c\neq b$ and $c\neq b'$, we cannot have a triangle of the form $\{ab,a'c,a'b'\}$. Thus, $\Xi^D_{ab}(a') = 0$.

If $x\neq a'$, then $D(ab,a'b')$ never appears in $\Xi^D_{ab}(x)$. Moreover, $a$, $x$, and $a'$ are all distinct, so $\{ab,xc,a'b'\}$ can never be a triangle in $\Gamma$. Thus, $\Xi^D_{ab}(x) = 0$.

\underline{Case 3: $w=ab'$.} In this case the edge $(ab,ab')$ is a vertical edge. Thus, $$D(ab,xc) = \begin{cases} -1, & \{ab,xc,ab'\} \in \Gamma\\ |\{a' \st \{ab,ab', a'b'\} \in \Gamma\}|, & xc = ab'\\ 0, & \text{otherwise}\end{cases}.$$

By assumption, we choose $x \in N_g(a)$, so $x\neq a$. Thus the middle case can never occur in $\Xi_{ab}^D(x)$. Furthermore, for all $c \in N_H(b)\setminus\{b'\}$, $D(ab,xc) = 0$ --- in this case, $(ab,xc)$ always falls into the third case of $D(ab,xc)$.

 Now, for any fixed $x$, there are only two triangles of the form $\{ab,ab',xc\}$ that can exist in $\Gamma$. These two triangles are $\{ab,ab',xb\}$, and $\{ab,ab',xb'\}$. However, $\{ab,ab',xb,xb'\}$ is a square in $G\times H$. Since $\Gamma$ is a fixed triangulation of $G\times H$, $\Gamma$ contains exactly one of the edges $(ab,xb')$ and $(ab',xb)$. Thus, exactly one of $D(ab,xb)$ and $D(ab,xb')$ is zero, and the other is $-1$. Thus, $\Xi_{ab}^D(x) = -1$ for any choice of $x \in N_G(a)$.

\underline{Case 4: $w=a'b$.} In this case, $$D(ab,xc) = \begin{cases} -1, & \{ab,a'b,xc\} \in \Gamma\\ |\{b' \st \{ab, a'b, a'b'\} \in \Gamma\}|, & xc = a'b\\ 0, & \text{otherwise}\end{cases}.$$

We note that if $x\neq a'$, then neither of the first two cases occurs in $\Xi^D_{ab}(x)$. Since $x\neq a'$, $\{ab,a'b,xc\}$ can never be a triangle, and neither will $D(ab,a'b)$ appear in our summation. Thus, $\Xi^D_{ab}(x) = 0$.

If $x=a'$, then $D(ab,a'b) = |\{b' \st \{ab, a'b, a'b'\}\in \Gamma\}|$, and $$\sum_{\substack{c \in N_H(b):\\ (ab,a'c) \in E(\Gamma)}}D(ab,a'c) = -1 \cdot |\{c \st \{ab,a'b,a'c\} \in \Gamma\}|.$$ Thus, $$\sum_{\substack{c \in V(H):\\ (ab,a'c) \in E(\Gamma)}} D(ab,a'c) = D(ab,a'b) + \sum_{\substack{c \in N_H(b):\\ (ab,a'c) \in E(\Gamma)}}D(ab,a'c) = 0,$$ so $\Xi^D_{ab}(x)$ is again independent of $x$.

We note that for all of the cases above, analogous arguments would hold for $\Upsilon^D_{ab}(y)$.

\textbf{Part 2:} Show that for any $v=(a,b) \in V(G\times H)$, the equations $$\left\{\Xi^D_{ab}(x) = \Xi^D_{ab}(x') : x, x' \in N_G(a)\right\} \cup \left\{\Upsilon_{ab}^D(y) = \Upsilon_{ab}(y') : y,y'\in N_H(b)\right\},$$ which we call \textbf{balancing conditions}, span the left kernel of $M_v$ over $\Q$ (i.e., they generate all $\Q$-linear relations on its rows).

We claim that there are at least $\deg(a) + \deg(b) - 2$ linearly independent balancing conditions. Suppose that $\{v_1,\dots,v_n\}$ is the set of neighbors in $G$ of $a$ and $\{w_1,\dots, w_m\}$ is the set of neighbors of $b$ in $H$. Then 
\begin{align*}
\sum_{c \in V(H)} D(v_1c,ab) &= \sum_{c \in V(H)} D(v_2c, ab) & \sum_{c \in V(G)} D(cw_1,ab) &= \sum_{c \in V(G)} D(cw_2, ab)\\ 
\sum_{c \in V(H)} D(v_1c,ab) &= \sum_{c \in V(H)} D(v_3c, ab) & \sum_{c \in V(G)} D(cw_1,ab) &= \sum_{c \in V(G)} D(cw_3, ab)\\ 
&\;\;\vdots &\;\;\vdots\\ 
\sum_{c \in V(H)} D(v_1c,ab) &= \sum_{c \in V(H)} D(v_nc, ab) & \sum_{c \in V(G)} D(cw_1,ab) &= \sum_{c \in V(G)} D(cw_m, ab)\end{align*} is a collection of $\deg(a) +\deg(b) - 2$ linearly independent linear relations among the values of $D$ on the edges of $\Gamma$ --- observe that for $i > 1$, the term $D(v_ic,ab)$ occurs only in the $(i-1)$st equation in the left-hand column, and for $j>1$, the term $D(cw_j,ab)$ occurs only in the $(j-1)$st equation in the right hand column.

Recall that the rows of $M_v$ are indexed by the edges of the graph star $E_{\Gamma}(v)$ of $v$. Thus, \begin{align*}\text{rank}(M_v) &\leq|E_{\Gamma}(v)| - (\deg(a) + \deg(b) -2)\\ &= |E_{\Gamma}(v)\cap\Diag(\Gamma)| =: \delta(v).\end{align*}

We will construct a submatrix of $M_v$ with $\delta(v)$ rows, and show that some $\delta(v) \times \delta(v)$ minor of this submatrix does not vanish. This will show that $\text{rank}(M_v) \geq \delta(v)$.

\underline{Case 1: $v$ is contained in at least one diagonal edge.} Let $\{E, E_1,\dots, E_m\}$ be the set of diagonal edges containing $v$, and let $U$ and $H$ be the vertical and horizontal edges of the square that contains $E$. Since $E$ contains $v$, $U$ and $H$ must contain $v$ as well. Let $S_v$ be the submatrix of $M_v$ consisting of the rows labeled by $\{E,U,H,E_1,\dots,E_m\}$, in that order. Let $\{d, u, h, d_1,\dots,d_m\}$ be the neighbors of $v$ contained in the edges $\{E,U,H, E_1,\dots, E_m\}$, respectively, and then define $R_v$ to be the submatrix of $S_v$ consisting of the columns of $S_v$ labeled by $\{d, u, h, d_1,\dots, d_m\}$, in that order.

We see that $R_v$ is a $\delta(v) \times \delta(v)$ square matrix, with the following block form:

$$R_v = \left(\begin{array}{cc} A & B\\ B^{T}& I\end{array}\right)$$
Let $x$ be a vertex and $Q$ be an edge. Then, $$R_v(Q,x) = \begin{cases}\alpha(Q,x), & x \in Q\\ -1, & Q \in \link_{\Gamma}(x)\\ 0, & \text{else}\end{cases},$$ so block $A$ has the form 
$$A = \left(\begin{array}{ccc}
1 & -1 & -1\\
-1 & \alpha(U,u) & 0\\
-1 & 0 & \alpha(H,h)\end{array}\right).$$

On the other hand, $I$ is an identity matrix. $E_i$ and $E_j$ share no common triangles when $i\neq j$ (so $E_i$ is never in $\link(d_j)$ or vice-versa), and $\alpha(d_i,E_i) = 1$ for all $i$ by definition of $\alpha$ for diagonal edges.

Since $E$ is the only diagonal edge that shares a square with both $U$ and $H$, we see that $B$ has the following form:

$$B = \begin{blockarray}{ccccccccccc}
& E_1 & \dots & E_k & E_{k+1} & \dots & E_{k+\ell} & E_{k+\ell+1} & \dots & E_m\\
\begin{block}{c(cccccccccc@{\hspace*{5pt}})}
    d & 0 & \cdots & 0 & 0 & \cdots & 0 & 0 & \cdots & 0 &\\
    u & -1 & \cdots & -1 & 0 & \cdots & 0 & 0 & \cdots & 0 &\\
    h & 0 & \cdots & 0 & -1 & \cdots & -1 & 0 & \cdots & 0 &\\\end{block}
\end{blockarray}\;,$$

where $\{E_1, \dots, E_k\}$ are the diagonal edges containing $v$ that are in $\link(u)$ and $\{E_{k+1}, \dots, E_{k+\ell}\}$ are the diagonal edges containing $v$ that are in $\link(h)$.

Now, \begin{align*}R_v =  \left(\begin{array}{cc} A & B\\ B^{T}& I\end{array}\right) &=  \left(\begin{array}{cc}I & B \\ 0 & I\end{array}\right)\left(\begin{array}{cc} A - BB^{T} & 0\\ B^{T} & I\end{array}\right),\\ \therefore \det\left(\begin{array}{cc}A & B\\ B^{T}& I\end{array}\right) &= \det(A-BB^{T}).\end{align*}

We see $$BB^{T} = \left(\begin{array}{ccc}0 & 0 & 0\\ 0 & k & 0\\ 0 & 0 & \ell\end{array}\right),$$ where $k= \#\{i : \; E_i \in \link(u)\}$, and $\ell = \#\{i : \;E_i \in \link(h)\}$. Thus, $$A - BB^{T} = \left(\begin{array}{ccc}1 & -1 & -1\\ -1 & \alpha(U,u) - k & 0\\ -1 & 0 & \alpha(H,h) - \ell\end{array}\right),$$ so $\det(A - BB^{T}) = (\alpha(U,u) - k)(\alpha(H,h) - \ell) - (\alpha(U,u) - k) - (\alpha(H,h) - \ell)$. By definition, $$\alpha(U,u) = |\{\sigma \in F(\Gamma) : U \in \sigma, u \notin \diag(\sigma)\}|,$$ i.e. the number of triangles of the form shown below.

\begin{center}
\includegraphics{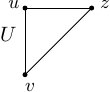}
\end{center}

The set of such triangles is in bijection with $\{e \in \Diag(\Gamma): \; v\in e, e\in \link(u)\}$. On the other hand, $k = |\{Q \in E_{\Gamma}(v) \cap \Diag(\Gamma) : Q \in \link_{\Gamma}(u), Q\neq E\}|$. Thus, $\alpha(U,u) = k + 1$, so $\alpha(U,u) - k = 1$. By an analogous argument, $\alpha(H,h) - \ell = 1$. Thus, $\det(R_v) = -1 \neq 0$.

\underline{Case 2: $v$ is not contained in any diagonal edges.} In this case, $\delta(v) = 2$, so we need to find some nonvanishing $2\times 2$ minor of $M_v$. Let $U$ be a vertical edge containing $v$, and let $H$ be a horizontal edge containing $v$. We write $U = vu$ and $H = vh$, and we let $R_v$ be the $2\times 2$ submatrix of $M_v$ whose columns are indexed by $v$ and $h$ respectively, and whose rows are indexed by $U$ and $H$, respectively. Then, $$R_v = \left(\begin{array}{cc} \alpha(U,v) & -1\\ \alpha(H,v) & \alpha(H,h)\end{array}\right).$$

Now, $\alpha(H,h)=0$, because by assumption, $v$ is not contained in any diagonal edge, so every triangle $\sigma$ containing $H$ must be of the form:

\begin{center}
\includegraphics{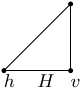}
\end{center}

where $h\in\diag(\sigma)$. On the other hand, $\alpha(H,v)>0$, since $\alpha(H,v) + \alpha(H,h) = \deg(H) > 0$ (recall that $\deg(H)$ is the number of facets containing $H$). Thus, $\det(R_v) = \alpha(U,v)\cdot 0 - (-1)(\deg(H)) = \deg(H)$. So $\text{rank}M_v \geq 2$, as desired. \end{proof}

The characterization of $\Q$-Cartier divisors given by balancing conditions is useful both as a technical tool (as will be demonstrated in later proofs), and as a means of making tropical $\Q$-Cartier divisors more understandable. We have already seen that the principal divisors on a weak tropical complex are the vectors in the column span of an easily-constructed matrix, and the balancing conditions allow us to construct a matrix whose kernel consists of the $\Q$-Cartier divisors.

\begin{example} Let $G = H = P_2$, the path with two edges. The Laplacian $L(G)$ of $G$ is 
$\left[\begin{array}{ccc}
-1 & 1 & 0\\
1 & -2 & 1\\
0 & 1 & -1	
\end{array}\right]$, which has rank $2$, so $\text{Coker}(L(G))$ is given by $\Z^3/\text{span}([-1,1,0],[0,1,-1]) \cong \Z$. Thus, $\Pic(G) \cong \Pic(H) \cong \Z$.
\begin{center}
\includegraphics{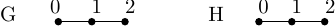}
\end{center}
 
The triangulation $\Gamma$ of $G\times H$ is shown below, with edges labeled by the letters $a$ through $p$:
\begin{center}
\includegraphics{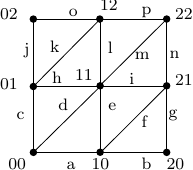}
\end{center}

A linear combination $D$ of edges must satisfy the following equations in order to be a $\Q$-Cartier divisor on $\Gamma$:
\begin{align*}D(a) &= D(b)+D(f)\\
D(h)+D(d) &= D(i)+D(m)\\
D(o)+D(k)&=D(p)\\
D(c)&=D(j)+D(k)\\
D(d)+D(e) &=D(l)+D(m)\\
D(f)+D(g) &= D(n)
\end{align*}

We rearrange all of these equations so that one side is equal to $0$, and hence can express a $\Q$-Cartier divisor $D$ as a vector in the kernel of the $6\times 16$ matrix $C$ whose columns are labeled by the edges of $\Gamma$ and whose rows are the characteristic vectors of the equations above.

We can also express the principal divisors on $\Gamma$ in terms of a matrix: the ridges of $\Gamma$ are the edges of $\Gamma$, and a $PL$-function on $\Gamma$ is uniquely determined by the values it takes on the vertices of $\Gamma$. We define a matrix $P$ with columns indexed by vertices of $\Gamma$ and rows indexed by edges of $\Gamma$, with the column corresponding to a vertex $v$ given by $\Div(\phi_v)$.

Thus, the divisor class group of $\Gamma$ can be viewed as $\ker(C)/\text{Im}(P)$. Using a computer algebra system (in this instance, Sage \cite{sage}), we see that $\text{Cl}(\Gamma) \simeq \Z^2$. Note that $\Pic(G) \times \Pic(H) \cong \Pic(\Gamma) \cong \Z^2$, by the main theorem of this paper (\ref{MainTheorem}) since $G$ and $H$ are both trees. In fact, Proposition \ref{eCart} implies that $\text{Cl}(\Gamma)$ and $\Pic(\Gamma)$ coincide in this case. \end{example}

\begin{remark} The balancing conditions from this section have a similar structure to the definition of a \emph{harmonic morphism} (see \cite[Section 2]{Hyperelliptic}). Briefly, if $G$ and $H$ are graphs, a \emph{morphim} 
$$\phi: (V(G)\sqcup E(G)) \to (V(H) \sqcup E(H))$$ 
(for convenience, we write $\phi: G \to H$) is a set map such that
\begin{enumerate}
\item $\phi(V(G)) \subseteq V(H)$
\item if $x \in V(G)$, $e \in E(G)$ with $x \in e$, either $\phi(x) = \phi(e)$, or $\phi(e) \in E(H)$ and $\phi(x) \in \phi(e)$.
\end{enumerate} 

A morphism $\phi:G \to H$ is \emph{harmonic} if, for all $x \in V(G)$ and $y \in V(H)$ such that $\phi(x) = y$, the quantity
$$|\{e \in E(G) \st x \in e, \phi(e) = e'\}|$$
is independent of the choice of edge $e'$ containing $y$.

Now, let $G$ and $H$ be simple graphs, let $\Gamma$ be a triangulation of $G \times H$ as in this section, and let $S$ be the $1$-skeleton of $\Gamma$ (i.e. the graph consisting of the vertices of $\Gamma$ and the edges of $\Gamma$). We observe that there is only one graph morphism from $S$ to $G$ that acts as projection onto the first coordinate when restricted to the vertex set $V(S) = V(G) \times V(H)$.

Indeed, suppose that $v \in V(G), w \in V(H)$, and that $\phi_G$ is a graph morphism from $S$ to $G$ with $\phi_G(v,w) = v$. If $e \in E(S)$ is a vertical edge of $\Gamma$ containing $(v,w)$, then $\phi_G(e) = \phi(v,w) = v$, since the other endpoint of $e$ is also mapped to $v$ under $\phi$ (and by assumption $G$ contains no loops). If $e \in E(S)$ is a horizontal or diagonal edge of $\Gamma$ that contains $(v,w)$, then the other endpoint of $e$ is mapped by $\phi_G$ to some neighbor $v'$ of $v$, so $e$ must be mapped to the edge $vv' \in E(G)$. Similarly, there is a unique graph morphism $\phi_H$ from $S$ to $H$ that restricts to projection onto the second coordinate on the vertex set of $S$.

The morphisms $\phi_G$ and $\phi_H$ are not harmonic in general. However, the balancing conditions for a $\Q$-Cartier divisor $D$ on $\Gamma$ are equivalent to saying that for any vertex $v \in V(G)$ and $w \in V(H)$, the sums
$$\sum_{\{r \in E(S) \st (v,w) \in r, \phi_G(r) = e\}} D(r)$$
and
$$\sum_{\{r \in E(S) \st (v,w) \in r, \phi_H(r) = f\}} D(r)$$
are independent of the choice of $e \in E(G)$ with $v \in E$ and the choice of $f \in E(H)$ with $w \in f$. Thus, we see that the divisor $$C = \sum_{r\in E(\Gamma)}[r]$$ is $\Q$-Cartier if and only if the maps $\phi_G$ and $\phi_H$ are harmonic.\end{remark}

\section{Cartwright's Conjecture}

Let $G$ and $H$ be graphs, and $\Gamma$ a triangulation of $G\times H$. Cartwright defined a map $\beta: \Div(G) \times \Div(H) \to \Div(\Gamma)$ \cite{Cartwright_Email} as follows. Let $C$ be a divisor on $G$ and $D$ be a divisor on $H$. Then, $$C = \sum_{v \in V(G)}C(v)[v]$$ and $$D = \sum_{w\in V(H)}D(w)[w],$$ so we define $$\beta(C,D):=\displaystyle\sum_{v\in V(G),r\in E(H)}C(v)[(v,r)] + \sum_{e \in E(G), w \in V(H)}D(w)[(e,w)],$$ where $(v,r)$ and $(e,w)$ are edges in $\Gamma$. We observe that $\beta$ is injective.

\begin{proposition} The map $\beta$ sends principal divisors to principal divisors.\end{proposition}

\begin{proof} The principal divisors on a graph are the vectors in the column span of the Laplacian matrix of that graph. For $w \in V(G)$, let $L_w$ be the column of the Laplacian of $G$ corresponding to $w$, and for $v \in V(H)$, let $L_v$ be the column of the Laplacian of $H$ corresponding to $v$. It is clear that $\beta$ is linear on $\Z\{V(G) \sqcup V(H)\}$, so in order to show that $\beta(\Prin(G) \times \Prin(H)) \subseteq \Prin(\Gamma)$, it suffices to show that $\{\beta(0,L_{v}) \st v \in V(H)\}$ and $\{\beta(L_{w},0) \st w \in V(G)\}$ are sets of principal divisors in $\Gamma$.

Fix some vertex $v \in H$. By the definition of $\beta$, we know that $$\beta(0,L_v) = \sum_{e \in E(G)}\deg(v)[(e,v)] - \sum_{ v' \in N_H(v)}\sum_{e \in E(G)}\text{adj}(v,v')[(e,v')],$$ where $\text{adj}(v,v')$ is the number of edges between $v$ and $v'$ in $H$.

We claim that \begin{equation}\label{importantthing}\beta(0,L_v) = \sum_{w \in V(G)}P(w,v),\end{equation} where $P(w,v) = \Div(\phi_{wv})$, and $\phi_{wv}$ is the PL function that is $1$ at vertex $wv$ and $0$ at all other vertices. We know that \begin{equation}\label{stupidthing}\sum_{w \in V(G)}P(w,v) = \sum_{w \in V(G)}\left(\underbrace{\sum_{r \in \link_{\Gamma}(wv)} -[r]}_{A} + \underbrace{\sum_{\substack{r \in \Diag(\Gamma):\\ wv \in r}}[r]}_B + \underbrace{\sum_{\substack{r \notin \Diag(\Gamma):\\ wv \in r}}\alpha(r,wv)[r]}_C\right).\end{equation}

For every $r \in E(\Gamma)$ we will show that the coefficient of $[r]$ in $\beta(0,L_v)$ is equal to the coefficient of $[r]$ in the right-hand side of \eqref{stupidthing}. We treat the cases of diagonal, vertical, and horizontal edges separately.

\underline{Case 1: $r\in \Diag(\Gamma)$.} We know that the coefficient in $\beta(0,L_v)$ of $[r]$ is $0$. On the other hand, in $P(w,v)$, $[r]$ can only occur in terms $A$ or $B$. There is exactly one square containing $r$ (since $r$ is a diagonal edge), and it has two possible forms:

\begin{center}
\includegraphics{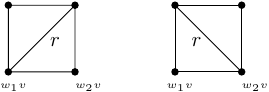}
\end{center}

Thus, we get a contribution of $-[r]$ either from $P(w_1,v)$ (in the second case), or from $P(w_2,v)$ (in the first case), and a contribution of $[r]$ from the other, so the coefficient of $[r]$ on the right-hand side of \eqref{stupidthing} is $0$.

\underline{Case 2: $r$ is a vertical edge of $\Gamma$.} We observe that $\beta(0,L_v)(r) = 0$. For any fixed $w$ in $V(G)$, $[r]$ can only occur in terms $A$ or $C$ in Equation \eqref{stupidthing}. 

Consider a square in $\Gamma$ containing $r$. Again, it has exactly two possible forms:

\begin{center}
\includegraphics{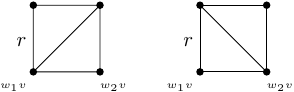}
\end{center}

The left square, $S$, contributes $0$ to the coefficient of $[r]$, because $w_1v$ is on $\diag(S)$, so $S$ not counted in $\alpha(r,w_1v)$, and $r \notin \link(w_2v)$. The right square, $T$, contributes $1$ to $\alpha(r,w_1v)$ in sum $C$ of $P(w_1,v)$ since $w_1v$ is not on the diagonal of $T$, but it contributes $-1$ to $\alpha(r,w_2v)$ in sum $A$ of $P(w_2,v)$, since $r \in \link_{\Gamma}(w_2v)$. Thus, the coefficient of $[r]$ in $\displaystyle\sum_{w \in V(G)}P(w,v)$ is $0$.

\underline{Case 3: $r$ is a horizontal edge of $\Gamma$.}

\underline{Case 3a: $wv \notin r$, $r \notin \link_{\Gamma}(wv)$ for any $w \in V(G)$.} The coefficient of $[r]$ in equation \eqref{stupidthing} must be $0$. Furthermore, the coefficient of $[r]$ in $\beta(0,L_v)$ is also $0$, by definition.

\underline{Case 3b: $wv \notin r$, $r \in \link_{\Gamma}(wv)$ for some $w \in V(G)$.} We can write $r = (e,v')$ for $e \in E(G)$, $v' \in N_H(v)$. Each square containing $r$ can have one of two forms:

\begin{center}
\includegraphics{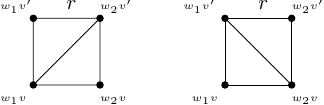}
\end{center}

For each $w \in V(G)$ such that $r \in \link_{\Gamma}(wv)$, the term $[r]$ only appears in sum $A$ of $P(w,v)$. For every edge in $H$ connecting $v$ and $v'$, there is a square of the form $\{w_1v,w_2v, w_1v', w_2v'\}$, which contributes $-1$ to the coefficient of $[r]$ --- $r$ is in exactly one of $\link_{\Gamma}(w_1v)$ or $\link_{\Gamma}(w_2v)$. Thus, the coefficient of $[r]$ in the right-hand side of equation \eqref{stupidthing} is $-\text{adj}(v,v')$, as we desired.

\underline{Case 3c: $w_1v \in r$ for some $w_1 \in V(G)$.} Each square in $\Gamma$ containing $r$ has one of the two following forms:

\begin{center}
\includegraphics{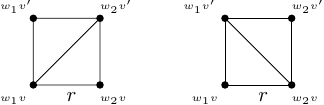}
\end{center}

In equation \eqref{stupidthing} $v$ is fixed, so $[r]$ can appear only in $P(w_1,v)$ and $P(w_2,v)$. Furthermore, $[r]$ appears only in sum $C$. Each square $S$ containing $r$ contributes $1$ to exactly one of $\alpha(r,w_1v)$ or $\alpha(r,w_2v)$ --- the diagonal edge of $S$ contains exactly one of $(w_1,v)$ and $(w_2,v)$. Thus, the coefficient of $[r]$ in Equation \eqref{stupidthing} is equal to the number of squares in $\Gamma$ containing $r$. 

A square in $\Gamma$ always has the form $e\times f$, where $e \in E(G)$, $f \in E(H)$. Thus, every square in $\Gamma$ containing $r$ has the form $(w_1,w_2) \times f$, where $f \in E(H)$ is some edge containing $v$. The number of such edges $f$ is equal to $\deg(v)$, so the number of squares containing $r$ is equal to $\deg(v)$. This means that the coefficient of $[r]$ in equation \eqref{stupidthing} is $\deg(v)$, as desired.

Finally, it is clear that if we switch the roles of $G$ and $H$, the same argument gives us that $\beta(L_w,0)$ is a principal divisor on $\Gamma$.\end{proof}

\begin{proposition}If $C$ is a divisor on $G$ and $D$ is a divisor on $H$, then $L:=\beta(C,D)$ is a Cartier divisor on $\Gamma$.\end{proposition}
\begin{proof} Fix a vertex $(v,w)$ of $\Gamma$. We wish to show that there is some principal divisor $P$ on $\Gamma$ that agrees with $L$ on all of the edges of $\Gamma$ containing $(v,w)$. By the definition of $\beta$, we know that for all $e \in E(G)$ $L(e,w) := D(w)$ and that for all $f \in E(H)$, $L(v,f) := C(v)$. We also know that $L \equiv 0$ on all of the diagonal edges containing $(v,w)$.

Now, since every divisor on a graph is Cartier, there exist principal divisors $P_1$ on $G$ and $P_2$ on $H$ such that $P_1(v) = C(v)$ and $P_2(w) = D(w)$. Clearly, $P:=\beta(P_1,P_2)$ is a divisor on $\Gamma$ that agrees with $L$ on all of the edges of $\Gamma$ containing $(v,w)$. We have already showed that $\beta: \Prin(G) \times \Prin(H) \to \Prin(\Gamma)$, so we are done.\end{proof}

Since every divisor on a graph is a Cartier divisor, we see that $\beta: \Cart(G) \times \Cart(H) \to \Cart(\Gamma)$. 

\begin{proposition}\label{eCart} The image of $\beta$ is $\{P \in \Cart(\Gamma) \st P(e) = 0 \Forall e \in \Diag(\Gamma)\}$.\end{proposition}
\begin{proof} The $\subseteq$ direction is the result of the previous proposition. For the $\supseteq$ direction, let $P$ be a Cartier divisor on $\Gamma$ that assigns the value $0$ to every diagonal edge and let $(a,b)$ be a vertex in $\Gamma$. Equation \eqref{balancing} says $\Xi_{ab}^{P}(x) = D(xb,ab)$ is independent of the choice of $x \in N_G(a)$. This tells us that $P$ is constant on $\Star_G(a) \times \{b\}$. If $a' \in V(G)$, the same argument says that $P$ is constant on $\Star_G(a')\times \{b\}$. Since $G$ is connected, this implies that $P$ is constant on $G \times \{b\}$. We call this quantity $\text{Horiz}(b)$.

A similar argument using $\Upsilon_{ab}^{P}$ shows that $P$ is constant on $\{a\} \times H$. We call this quantity $\text{Vert}(a)$. We write $$C = \sum_{a \in V(G)}\text{Vert}(a)[a]$$ and $$D = \sum_{b \in V(H)}\text{Horiz}(b)[b],$$ and it is then clear that $\beta(C,D) = P$.\end{proof}

We now come to the main theorem.

\newtheorem*{MainTheorem}{Theorem \ref{MainTheorem}}
\begin{MainTheorem} The map $\gamma: \Pic(G) \times \Pic(H) \to \Pic(\Gamma)$ induced by $\beta$ is injective, and is surjective if either of $G$ or $H$ is a tree.\end{MainTheorem}

This was conjectured by Cartwright at the AIM workshop ``Generalizations of chip-firing and the critical group'' \cite{Cartwright_Email}, and was motivated by analogous results in algebraic geometry --- if $X_1$ and $X_2$ are varieties, $\Pic(X_1) \times \Pic(X_2) \hookrightarrow \Pic(X_1 \times X_2)$, and $\Pic(X_1 \times P^1) = \Pic(X_1)$ (see, for instance, \cite[Theorem 1.7]{Ischebeck}).

\begin{proposition}The map $\gamma$ is always injective.\end{proposition}
\begin{proof}
We let $\text{ePrin}(\Gamma)$ be the abelian subgroup of $\text{Prin}(\Gamma)$ consisting of principal divisors that are $0$ on all diagonal edges. We claim that $\text{ePrin}(\Gamma)$ is the image of $\text{Prin}(G)\times \text{Prin}(H)$ under $\beta$. We know that $\beta$ is injective, so $\text{rank}(\beta(\text{Prin}(G)\times\text{Prin}(H)))=|V(G)| + |V(H)| - 2$.

On the other hand, if $C$ is a principal divisor, we can write $C = \Div(\phi)$ for some PL function $\phi$. Consider a square $S=\{a,b,c,d\}$ as pictured below, with $\diag(S) = r$. By equation \eqref{divisorequation}, $C(r) = \phi(a) +\phi(c) - \phi(b) - \phi(d)$.

\begin{center}
\includegraphics{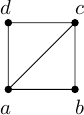}
\end{center}

If $C \in \text{ePrin}(\Gamma)$, we have $\phi(a) + \phi(c) = \phi(b) + \phi(d)$. Recall that every PL function $\phi$ can be viewed as a vector in $\Z^{V(\Gamma)}$. Let $W$ be the vector space of PL functions on $\Gamma$ that map to divisors in $\text{ePrin}(\Gamma)$ under $\Div$. Then, if $\phi \in W$, we see that every square in $\Gamma$ gives rise to a linear constraint on the entries in the vector $\phi$. We observe that $\ker(\Div)\subseteq W$.

We claim that the codimension of $W$ is at least $(|V(G)|-1)(|V(H)|-1)$. To see this, let $T_1$ be a spanning tree in $G$, and $T_2$ be a spanning tree of $H$. Then $T_1\times T_2$ gives rise to a subcomplex of $\Gamma$ with $(|V(G)|-1)(|V(H)|-1)$ squares (since $T_1$ has $|V(G)|-1$ edges and $T_2$ has $|V(H)|-1$ edges). Fix a root $v_1$ of $T_1$, and take some linear extension of the partial order on the vertices of $T_1$ arising from this choice of root. Fix a similar linear order on the vertices of $T_2$. Now, these linear orders give rise to a lexicographic ordering on the squares in $T_1 \times T_2$, with the property that the $i$th square $S_i$ includes a vertex that is not in $S_1\cup\cdots\cup S_{i-1}$. This means that the relation arising from each square is linearly independent of all prior relations. Thus, $\text{codim}(W) \geq (|V(G)|-1)(|V(H)|-1)$.

We know that $\dim(\text{ePrin}(\Gamma)) = \dim(W) - \dim(\ker(\Div))$. We also know that $\text{codim}(W) = |V(\Gamma)| - \dim(W)$, so $\dim(W) = |V(\Gamma)| - \text{codim}(W)$. Thus, \begin{align*}\dim(\text{ePrin}(\Gamma)) &= |V(\Gamma)| - \text{codim}(W) - \dim(\ker(\Div))\\ &\leq |V(G)||V(H)| - (|V(G)|-1)(|V(H)-1|) - \dim(\ker(\Div))\end{align*}

Furthermore, we know that the all $1$'s vector is in $\ker(\Div)$, so $\dim(\ker(\Div)) \geq 1$. Thus, we can write \begin{align*}\dim(\text{ePrin}(\Gamma)) &\leq |V(G)||V(H)| - (|V(G)| -1)(|V(H)| -1) -1\\ &\leq |V(G)| + |V(H)| -2.\end{align*} If $a: \Prin(G) \times \Prin(H) \to \text{Prin}(\Gamma)$ is the restriction of $\beta$ to $\Prin(G) \times \Prin(H)$, we see that $\text{rank}(\text{ePrin}(\Gamma))= \text{rank}(\im(a))$ (as free abelian groups).

We claim that $\im(a)$ is actually equal to $\text{ePrin}(\Gamma)$, i.e., that $\text{coker}(a)$ is torsion-free. Let $D \in \text{ePrin}(\Gamma)$ with $kD \in \im(a)$ for some $k \in \N$. Then $kD = a(B,C)$, where $B,C$ are in $\Prin(G)$ and $\Prin(H)$, respectively. However, by the definition of $a$, we know that the value of $kD$ on every (nondiagonal) edge of $\Gamma$ is either the value of $B$ on some vertex of $G$ or from the value of $C$ on some vertex of $H$. Thus, $\frac{B}{k}$ and $\frac{C}{k}$ are integer-valued, and since $B$ and $C$ are in $\Prin(G)$ and $\Prin(H)$, respectively, so are $\frac{B}{k}$ and $\frac{C}{k}$. Thus, $a(\frac{B}{k},\frac{C}{k}) = D$, so $a(\Prin(G) \times \Prin(H)) = \text{ePrin}(\Gamma)$.

Let $\gamma: \text{Pic}(G) \times \text{Pic}(H) \to \text{Pic}(\Gamma)$ be the map induced by $\beta$. We apply the Snake Lemma to the commutative diagram

$$\xymatrix{0 \ar[r] & \Prin(G) \times \Prin(H) \ar[r] \ar[d]^{a} & \Cart(G)\times \Cart(H) \ar[r] \ar[d]^{\beta} & \Pic(G) \times \Pic(H) \ar[d]^{\gamma} \ar[r] & 0\\ 0 \ar[r]& \Prin(\Gamma) \ar[r]^{i} & \Cart(\Gamma) \ar[r] & \Pic(\Gamma) \ar[r] & 0}$$ and we obtain an exact sequence

$$\xymatrix{0 \ar[r] & \ker(\gamma) \ar[r] & \text{coker}(a) \ar[r]^{i_*} & \text{coker}(\beta) \ar[r] & \text{coker}(\gamma)},$$

where $i_*$ is the map induced by the inclusion  $\Prin(\Gamma) \hookrightarrow \Cart(\Gamma)$. We will show that $i_*$ is injective, which will imply by exactness that $\ker(\gamma) = 0$. 

Let $[A], [B] \in \text{coker}(a)$, and suppose that $i_{*}([A]) = i_{*}([B])$. Then, $i(A) - i(B) \in \im(\beta)$, so $A-B$ is a Cartier divisor that is $0$ on all diagonal edges of $\Gamma$ by Proposition \ref{eCart}. In fact, $A-B$ is a principal divisor, so $A-B \in \text{ePrin}(\Gamma)=\im(a)$.\end{proof}

\begin{proposition}\label{SurjProof}The map $\gamma: \Pic(G) \times \Pic(H) \to \Pic(\Gamma)$ is surjective if at least one of $G$ or $H$ is a tree.\end{proposition}

\begin{proof}
First, we consider the case $H = K_2$, the complete graph on two vertices. Again, $\Gamma$ is a triangulation of $G\times K_2$. We let $w_1$ and $w_2$ be the vertices of $K_2$. Now, the diagonal edges of $\Gamma$ are in one-to-one correspondence with the edges of $G$: every square in $G \times K_2$ is of the form $r \times w_1w_2$, for $r \in E(G)$, and each square contains exactly one diagonal edge. We define an orientation on $G$ as follows. For each square, orient each edge of $G \times\{w_2\}$ so that it points toward the vertex containing the diagonal edge:

\begin{center}
\includegraphics{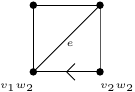}
\end{center}

Note, if $e$ is the diagonal edge of a square $S$ in $G\times K_2$, we let $\tilde{e}$ be the edge in $S \cap (G\times \{w_2\})$.

We now consider the coboundary matrix $M$ of $G$. This is a matrix whose rows are indexed by edges of $G$, and whose columns are indexed by vertices of $G$, with $$M(e,v) = \begin{cases}1, & v\in e, \text{ arrow pointing toward } v\\ -1, & v \in e, \text{ arrow pointing away from } v\\ 0, & v\notin e\end{cases}.$$ Let $P(\Gamma)$ be the matrix whose columns span the principal divisors of $\Gamma$. We let $M'$ be the submatrix of $P(\Gamma)$ whose rows correspond to the diagonal edges of $\Gamma$, and whose columns correspond to the vertices of $G \times \{w_2\}$. 

We claim that $M = M'$. First, we observe that $M'$ and $M$ have the same size --- the columns of $M'$ are labeled by the vertices in $G \times \{w_2\}$, and the rows of $M'$ are labeled by the diagonal edges of $\Gamma$ (which we know to be in bijection with the edges of $G$). Furthermore, suppose that $(v,w_2)$ is a vertex in $G\ \times \{w_2\}$, and that $vw_2$ has a diagonal edge $d$ in its link. Then, $M'(d,vw_2) = -\alpha(d,vw_2) = -1$. When we look at the corresponding square, we see that the edge $e$ in $G$ corresponding to $d$ is oriented away from $v$, so $M(e,v) = -1$. Similarly, suppose that $vw_2$ is contained in a diagonal edge $d$ in $\Gamma$. Then $M'(d,vw_2) = -\alpha(d,vw_2) = 1$, and when we look at the corresponding square, we see that the edge $e$ in $G$ is oriented toward $v$, so $M(e,v) = 1$. By the same argument, when $d$ is the diagonal edge in $\Gamma$ corresponding to an edge $e$ in $G$, $M(e,v) = 0 \iff M'(d,vw_2) = 0$.

By \cite[p. 31]{Combinatorial_Matrix_Theory}, we know that $M$ is totally unimodular, so the cokernel of $M$ is torsion-free. We also know from \cite[p. 168]{Godsil_Royle} that the rank of $M$ is $|V(G)| -1$. We observe that $M$ is is the transpose of the oriented incidence matrix of $G$, and so by \cite[Corollary 14.2.3]{Godsil_Royle}, the following relations span the left kernel of $M$. Let $\{v_0,\dots,v_{k-1}\}$ be a cycle in $G$, and let $M_{v_i,v_j}$ be the row of $M$ corresponding to the oriented edge $(v_i,v_j)$. Then, $\displaystyle\sum_{i=0}^{k-1}\mathcal{O}(v_i,v_{i+1})M_{v_i,v_j} = 0$ (with the index $i$ considered modulo $k$), where $$\mathcal{O}(v_i,v_{i+1}) = \begin{cases}1, & (v_i,v_{i+1}) \text{ is oriented toward } v_{i+1}\\ -1, & (v_i,v_{i+1}) \text{ is oriented toward } v_i \end{cases}.$$ 

Let $D$ be a Cartier divisor on $\Gamma$. Then, for every vertex $v_iw_2$ in $G\times \{w_2\}$, the horizontal balancing condition at $v_iw_2$ is: \begin{equation}\label{SurjBalance}D(v_iw_2,v_{i-1}w_1) + D(v_{i-1}w_2,v_iw_2) = D(v_iw_2,v_{i+1}w_1) + D(v_{i+1}w_2,v_iw_2),\end{equation} where we adopt the convention that $D(e) = 0$ if $e \notin E(\Gamma)$.

Therefore, the Cartier divisor $D$ satisfies $$\sum_{i=0}^{k-1}\left(D(v_iw_2,v_{i-1}w_1) + D(v_{i-1}w_2,v_iw_2)\right) = \sum_{i=0}^{k-1}\left(D(v_iw_2,v_{i+1}w_1) + D(v_{i+1}w_2,v_iw_2)\right).$$ We observe that $D(v_{i+1}w_2,v_iw_2)$ occurs exactly twice in this sum. Once on the right-hand side in the $i$th summand, and once on the left-hand side in the $(i+1)$st summand. Thus, we can cancel all such terms, and rewrite the equation to obtain: \begin{equation}\label{SurjBalanceSimplified}\sum_{i=0}^{k-1}D(v_iw_2,v_{i-1}w_1) = \sum_{i=0}^{k-1}D(v_iw_2,v_{i+1}w_1).\end{equation}

We observe that each diagonal edge $e= (v_iw_2,v_{i-1}w_1)$ (i.e., those on the left-hand side of equation \eqref{SurjBalanceSimplified}) gives rise to an orientation of the edge $\tilde{e} = (v_{i-1}w_2,v_iw_2)$ that points towards $v_iw_2$, and a diagonal edge $e = (v_iw_2,v_{i+1}w_1)$ (i.e., one on the right-hand side of equation \eqref{SurjBalanceSimplified}) gives rise to an orientation of the edge $\tilde{e} = (v_iw_2,v_{i+1}w_2)$ that points away from $v_iw_2$. This means that the values of $D$ on the diagonal edges of $G \times K_2$ must satisfy the linear relations on the rows of $M$ (when viewed as the coboundary matrix of a graph), so the diagonal part of $D$ must be in the $\R$-span of $M$ and hence in the $\Z$-span of $M$, by the total unimodularity of $M$.

Thus, every Cartier divisor $D$ on $\Gamma$ is linearly equivalent to a divisor $D'$ that is $0$ on all diagonal edges. Since the columns of $M'$ are indexed by vertices of the form $vw_2$, $v \in V(G)$, we can write $$D' = D - \sum_{v \in V(G)}k_v\Div(\phi_{vw_2}), \qquad k_v \in \Z,$$ as desired.

Now, let us assume that $G$ is a graph and $H$ is a tree. Fix a root $r$ of $H$, and, for every vertex $w \in H$, let $H_w$ be the subtree rooted at $w$. This choice of root in $H$ induces a partial order on $H$ so that every non-root vertex of $H$ has a unique parent. For the purposes of this part of the proof, if $\Gamma'$ is a subcomplex of $G \times H$, we let $\Gamma(\Gamma') := \Gamma|_{\Gamma'}$.

We claim that for every Cartier divisor $D$ on $\Gamma$, and every subtree $H'$ of $H$ that is rooted at $r$, we can find a Cartier divisor $D'$ that is linearly equivalent to $D$ and that vanishes on $\Diag(\Gamma(G\times H'))$. We prove this by induction on $|E(H')|$. 

If $|E(H')| = 0$, the statement is trivial. If $|E(H')| = 1$, we are done by the first part of the proof. 

Now, suppose that $|E(H')| = n$, and let $w$ be a leaf in $H'$, with $\ell$ the unique edge containing it. Then, $H'':=H'\setminus\{\ell\}$ is a tree with $n-1$ edges, so by induction there is a Cartier divisor $D_{\ell}$ that is linearly equivalent to $D$ and that vanishes on $\Diag(\Gamma(G\times H''))$. By the first part of the proof, there is a Cartier divisor $$D' = D_{\ell} + \sum_{i=1}^nk_i\Div(\phi_{v_iw}),$$ that vanishes on $\Diag(\Gamma(G\times \{\ell\}))$. We note that $\Div(\phi_{v_iw})$ vanishes on $\Diag(G\times H'')$ for all $i$, so $$D_{\ell}|_{\Diag(\Gamma(G\times H''))} = D'|_{\Diag(\Gamma(G\times H''))}.$$ Thus, $D'$ is a Cartier divisor that is linearly equivalent to $D$ and that vanishes on $\Diag(\Gamma(G\times H'))$. This concludes the induction.

Our induction has shown that every Cartier divisor on $\Gamma$ is linearly equivalent to one that is zero on all diagonal edges of $\Gamma$. By Proposition \ref{eCart}, every Cartier divisor that is $0$ on all diagonal edges of $\Gamma$ is in the image of $\beta$. So, let $[D] \in \Pic(\Gamma)$ be a linear equivalence class of Cartier divisors. Then, $[D] = [D']$, for some $D'$ that vanishes on all diagonal edges of $\Gamma$. We know that $D' = \beta(A,B)$, where $A$ and $B$ are Cartier divisors on $G$ and $H$ respectively, so $\gamma([A],[B]) = [D'] = [D]$. \end{proof}

Based on numerical evidence obtained using Sage \cite{sage}, we conjecture a slightly stronger result. Let $g(G)$ be the topological genus of $G$, i.e. the number of edges of $G$ in the complement of a spanning tree. Then:

\newtheorem*{MyConjecture}{Conjecture \ref{MyConjecture}}
\begin{MyConjecture}$\Pic(\Gamma) \cong \Pic(G) \times \Pic(H) \times \Z^{g(G)g(H)}$. \end{MyConjecture}

\bibliographystyle{amsplain}
\bibliography{biblio}






\end{document}